\documentclass[12pt,a4paper,reqno]{amsart}
\usepackage{amssymb}
\usepackage{amsmath,amsthm}
\usepackage{a4wide}
\usepackage{cancel}
\usepackage{cite}
\newtheorem{theorem}{Theorem}[section]
\newtheorem{lemma}[theorem]{Lemma}
\newtheorem{proposition}[theorem]{Proposition}

\theoremstyle{definition}
\newtheorem{definition}[theorem]{Definition}

\def\N{\mathbb N}

\newcommand{\R}{\mathbb{R}}

\newcommand{\RR}{\mathbb R}

 \newcommand{\bea}{\begin{eqnarray}}
 \newcommand{\eea}{\end{eqnarray}}

 \newcommand{\beas}{\begin{eqnarray*}}
 \newcommand{\eeas}{\end{eqnarray*}}

\theoremstyle{remark}
\newtheorem{remark}[theorem]{Remark}

\numberwithin{equation}{section}

\theoremstyle{remark}

\usepackage{color}

\usepackage{xcolor}

\title[Abelian and Tauberian Results for the FrHT]{Abelian and Tauberian Results for the Fractional Hankel Transform of Generalized Functions}

\author[S. Atanasova]{Sanja Atanasova}
\address{Faculty of Electrical Engineering and Information Technologies, Ss. Cyril and Methodius University\\  Rugjer Boshkovik 18\\
1000 Skopje\\ North Macedonia}
\email{ksanja@feit.ukim.edu.mk}
\author[S. Jak\v{s}i\'{c}]{Smiljana Jak\v{s}i\'{c}}
\address{Faculty of Foresty, University of Belgrade\\  Kneza Vi\v{s}eslava 1, \\ 11030 Belgrade \\ Serbia}
\email{smiljana.jaksic@sfb.bg.ac.rs}
\author[S. Maksimovi\'{c}]{Snje\v{z}ana Maksimovi\'{c}}
\address{Faculty of Architecture, Civil Engineering and Geodesy, University of Banja Luka\\ Bulevar vojvode Petra Bojovi\'{c}a 1A \\ 78000 Banja Luka\\ Bosnia and Hercegovina} \email{snjezana.maksimovic@aggf.unibl.org}
\author[S. Pilipovi\'{c}]{Stevan Pilipovi\'{c}}
\address{Faculty of Sciences and Mathematics, University of Novi Sad\\ Trg D. Obradovica 4\\ 21000 Novi Sad\\ Serbia}
\email{pilipovic@dmi.uns.ac.rs}

\begin{document}

\begin{abstract}
This paper aims to explore the quasiasymptotic behavior of distributions through the fractional Hankel transform. We present  Tauberian result that connects the asymptotic behavior of generalized functions in the Zemanian space with the asymptotics of their fractional Hankel transform. Additionally, we establish both the initial and final value theorems for the fractional Hankel transform of distributions.
\end{abstract}

\keywords{Fractional Hankel transform, distributions, Abelian and Tauberian theorems}


\subjclass[2020]{40E05, 46F05, 46F12, 47B35}
\maketitle

\section{Introduction}

The Hankel transform is an integral transform tailored for problems with circular or cylindrical symmetry, frequently used in applied mathematics, physics, and engineering. The zero-order transform describes diffraction of axially symmetric light beams, while higher-order versions apply to laser cavities with circular mirrors \cite{Pras}. Its extension to spaces of distributions has been explored in various works: Betancor and Rodríguez-Mesa studied distributions of exponential growth \cite{Betancor1}, while Zemanian extended the transform to both slowly and rapidly growing distributions \cite{Zem, Zem1, Zem2}. More  results on the Hankel transform in the setting of distribution spaces can be found in \cite{Betancor, Betancor1, Hankel1, Hankel2, 5, Ridenhour, Zem, Zem1, Zem2}.

The fractional Hankel transform (FrHT) generalizes the classical Hankel transform by replacing the standard Bessel kernel with one of fractional order. This concept was first introduced by Namias in \cite{NamiasH}.  The FrHT  is particularly useful in areas such as lens design, laser cavity analysis, and the study of wave propagation in media with a quadratic refractive index,  \cite{She, 13}. In \cite{55}, Kerr developed the theory of the FrHT in $L^2(\RR_+)$, and later extended it to the Zemanian spaces $\mathcal K_\mu(\RR_+)$ in \cite{5}. Additionally, Prasad and Singh investigated the FrHT of tempered distributions and introduced pseudo-differential operators involving the FrHT in \cite{Pras}.

The distribution theory is a powerful tool in applied mathematics.
In the study of distribution theory, one quickly realizes that distributions does not have point values, which leads to the introduction of the concept of quasiasymptotics as a extension of the concept of asymptotics in the framework of distributions (see \cite{VDZ, 7} and references therein).  Recent articles where the quasiasymptotics is explored in a form of structural theorems  of Abelian and Tauberian-type  for different integral transforms are \cite{APS, 6-2, KPSV, AMP, MAM}.

The main objective of this paper is to establish a Tauberian theorem, an initial value theorem, and a final value theorem for the FrHT on the Zemanian space $\mathcal{K}'_{\mu}(\RR_+)$. In Section \ref{Sec Abelian and Tauberian-type results} we present a Tauberian type result, where we study the quasiasymptotic boundedness and the quasiasymptotic behavior of a generalized function in $\mathcal K'_{\mu}(\R_+)$ under the corresponding Tauberian type conditions.
In Section \ref{Sec IVT and FVT}, we present the initial and final value theorems for the FrHT. These results provide insight into the behavior of the FrHT at the boundaries of its domain:
the initial value theorem allows the evaluation of the FrHT at the origin, while the final value theorem describes its behavior as the argument tends to infinity.
By establishing these theorems, we gain a deeper understanding of the asymptotic properties of the FrHT, which is particularly relevant in applications involving wave propagation and other phenomena with boundary conditions.



\section{Space $\mathcal{K}_{\mu}(\RR_+), \mu\geq  -1/2$}\label{Sec prostor K}

Let $m,k\in\N_0$, and $\mu\geq-1/2$. Then, following \cite{Zem}, we say that a smooth and complex-valued function
 $\varphi(x)$ on $\RR_+=(0,+\infty)$ belongs to the space $\mathcal{K}_{m,k}^{\mu}(\RR_+)$ if
$$\gamma_{m,k}^\mu(\varphi)=\sup_{x\in \RR_+}\Big|x^m(x^{-1}D)^k\big(x^{-\mu-\frac{1}{2}}\varphi(x)\big)\Big|<\infty.$$
$\mathcal{K}_{m,k}^{\mu}(\RR_+)$ is a linear space. 

Zemanian defined in \cite[Section 5.2]{Zem} the locally convex space $\mathcal{K}_{\mu}(\RR_+)$ as ( see \cite[(8), p.133]{Zem})
$$\mathcal{K}_{\mu}(\RR_+)=\bigcap_{m,k\in\N_0}\mathcal{K}_{m,k}^{\mu}(\RR_+),$$
with the topology  assigned to it by the norms
$$\gamma_r^\mu(\varphi)=\max_{0\leq m \leq r \atop 0\leq k\leq r}\gamma_{m,k}^\mu(\varphi), \  r\geq 0.$$
It is shown in \cite[Lemma 5.2-2]{Zem} that $\mathcal{K}_{\mu}(\RR_+)$ is complete, and therefore a Fr\'{e}chet space; $\mathcal{D}(\RR_+)$ is a subspace of $\mathcal{K}_{\mu}(\RR_+)$, but it is not dense in $\mathcal{K}_{\mu}(\RR_+)$. Furthermore, $\mathcal{K}_{\mu}(\RR_+)$ is a dense subspace of $\mathcal{E}(\RR_+)$.

The dual space of $\mathcal K_\mu(\RR_+)$ is the space $\mathcal K'_\mu(\RR_+)$ consisting of distribution with slow growth and specific behavior in an neighborhood of zero. $\mathcal K'_\mu(\RR_+)$ complete space.
We know that if $f\in\mathcal K'_{\mu}(\R_+)$, then there exist $r>0$ and $C>0$
such that
\begin{equation}\label{uslnep}
|\langle f,\phi\rangle|\leq C\gamma^\mu_r(\phi),\; \phi\in\mathcal K_{\mu}(\R_+).
\end{equation}

If  $f(x)$ is a locally integrable function on $\R_+$ such that $f(x)$ is of slow growth, as $x\rightarrow\infty$,
and $x^{\mu+\frac{1}{2}}f(x)$ is locally integrable on $0 < x < 1$, then $f(x)$ generates a regular
generalized function $f(x)$ on $\mathcal K'_\mu(\RR_+)$ by (see \cite[p.143]{Pras})
$$\langle f,\varphi\rangle=\int_0^\infty f(x)\varphi(x)dx, \quad \varphi\in \mathcal{K}_\mu(\RR_+) . $$

\section{Fractional Hankel transform}\label{FRHT}

The Hankel transform of order $\mu$ of a function $f\in L^1(\RR_+)$ is
defined by
\begin{equation*}\label{ht}
H_\mu f(\xi)=\int_0^\infty \sqrt{x\xi}J_\mu(x\xi) f(x)dx, \ \ \xi\in \RR_+,
\end{equation*}
where, $\mu\geq -1/2$ and
$$J_\mu(x)=\left(\frac{x}{2}\right)^{\mu}\sum_{n=0}^{\infty}\frac{(-1)^n}{n!\Gamma(\mu+n+1)}\left(\frac{x}{2}\right)^{2n}, \quad x\in \RR_+ $$
is the Bessel function of  first kind and order $\mu$. If $H_\mu f\in L^1(\RR_+)$, then the inverse Hankel transform is given by
$$ f(x)=\int_0^\infty \sqrt{x\xi}J_\mu(x\xi) H_\mu f(\xi)d\xi, \ \ x\in \RR_+.$$

The FrHT is a generalization of the  Hankel
transform. It is not surprising that the FrHT closely resembles the fractional Fourier transform, given that the two-dimensional Fourier transform is closely related to the Hankel transform. In fact, FrHT can be directly derived from the two-dimensional fractional Fourier transform.

Recall \cite[(1.3)]{Pras}, the FrHT with parameter $\alpha$ of $f(x)$ for $\mu\geq -1/2$ and $\alpha\in(0,\pi)$,  is defined by

\begin{equation*}\label{frht}
{H}_\mu^\alpha f(\xi)=
\int_{0}^\infty K_{\alpha}(x,\xi)f(x)dx,
  \end{equation*}
where
\begin{equation*}\label{kernel}
 K_{\alpha}(x,\xi)=\begin{cases}
 C_{\alpha,\mu} e^{-i(\frac{x^2+\xi^2}{2}c_1)}\sqrt{x\xi c_2}J_\mu(x\xi c_2), & \alpha\neq\frac{\pi}{2}\\
 \sqrt{x\xi}J_\mu(x\xi), & \alpha=\frac{\pi}{2}
 \end{cases},
   \end{equation*}
is the kernel of the FrHT,
 $c_1=\cot\alpha$, $c_2=\csc \alpha$,  $$C_{\alpha,\mu}=\frac{e^{i(1+\mu)(\frac{\pi}{2}-\alpha)}}{\sin\alpha}.$$
Note that ${H}_\mu^\alpha$ is of period $\pi$ in $\alpha$, and ${H}_\mu^0 f={H}_\mu^\pi f=f$ for $f\in \mathcal K_\mu(\RR_+)$.

It is proven in \cite[Theorem 3.1]{Pras} that the FrHT is a continuous linear mapping of $\mathcal K_\mu(\R_+)$ onto itself.

For $f\in \mathcal K_\mu(\RR_+)$, the inverse FrHT is defined in \cite[(1.4)]{5} as
\begin{equation}\label{inverse}
(H_{\mu}^{\alpha})^{-1}H_{\mu}^{\alpha}f(x)=f(x)=\int_{0}^\infty \overline{K}_{\alpha}(x,\xi)(H_\mu^\alpha f)(\xi)d\xi.
\end{equation}
Here,
$$\overline{K}_{\alpha}(x,\xi)=C_{\alpha, \mu}^{\star}e^{i(\frac{x^2+\xi^2}{2}c_1)}\sqrt{x\xi c_2}J_\mu(x\xi c_2), \, C_{\alpha,\mu}^{\star}=\overline{C_{\alpha,\mu}}\sin{\alpha}.$$
Let $\alpha, \beta\in\mathbb R$,  $f\in \mathcal K_\mu(\R_+)$. By \cite[Theorem 4.22]{5} we have $H_{\mu}^{\alpha+\beta}f(x)=H_{\mu}^{\alpha}H_{\mu}^{\beta}f(x).$ Moreover, by \cite[Theorem 4.23]{5}, $H_{\mu}^{\alpha}f(x)\to H_{\mu}^{\beta}f(x)$ if $\alpha\to\beta.$
The following formula connects the FrHT and the Hankel transform

\begin{equation}\label{FrHT veza HT}
H_\mu^\alpha f(\xi)=C_{\alpha,\mu}e^{-ic_1\xi^2/2}H_\mu(e^{-ic_1 x^2/2}f(x))(c_2\xi), \ x,\xi\in \RR_+,
\end{equation}

\noindent see  \cite[(3.3)]{Pras}.

The FrHT can be extended to the space of distributions $\mathcal K'_\mu(\R_+)$. More precisely, \cite[Definition 3.1]{Pras} stays that the FrHT $H_\mu^{\alpha} f$ of $f\in\mathcal K'_\mu(\R_+)$ is defined by
\begin{equation*}\label{frdirect}
\langle H_\mu^{\alpha} f,\varphi\rangle=\langle  f,H_\mu^{\alpha}\varphi\rangle, \quad \varphi\in\mathcal K_\mu(\RR_+),
\end{equation*}
and furthermore, \cite[Theorem 3.2]{Pras} states that FrHT is a continuous linear mapping of $\mathcal K'_\mu(\R_+)$ onto itself.


\section{The quasiasymptotic behavior and   boundedness of distributions}

We recall some known fact concerning the quasiasymptotic behavior of generalized functions, see \cite{VDZ,7,Vindas1,Vindas2,Vindas5}.
A positive real-valued function, measurable on an
interval $(0,A]$ (resp. $[A,+\infty )),$ $A>0$, is called \textit{a
slowly varying function} at the origin (resp. at infinity), if
\begin{equation} \label{limL1}
\lim_{\varepsilon \to 0^{+} } \frac{L(a\varepsilon
)}{L(\varepsilon )} =1\quad \Big(\ {\rm resp.} \lim_{\lambda \to
+\infty } \frac{L(a\lambda )}{L(\lambda )} =1\Big)\quad {\rm for \ \
each} \ a>0.
\end{equation}
Throughout the rest of the article, $L$ will consistently denote a positive  slowly varying function at the origin (resp. at infinity).
\begin{definition}
Let $f\in {\mathcal K_{\mu}'}(\RR_+)$. We say that the distribution $f$ has
the quasi-asymptotic behavior (\textit{the quasiasymptotics}) of
degree $m \in {\Bbb R}$  at a point $x_{0} \in {\Bbb R}$
with respect to $L$ in $\mathcal{K}_{\mu}'(\RR_+)$ if there exists $u\in \mathcal{K}_{\mu}'(\RR_+)$
such that for every test function $\varphi \in \mathcal{K}_{\mu}(\RR_+)$ the following limit holds:
\begin{equation} \label{qbeh_x_0}
\lim_{\varepsilon \to 0^{+} } \Big\langle \frac{f(x_{0} +\varepsilon
x)}{\varepsilon ^{m } L(\varepsilon )} ,\ \varphi (x)\Big\rangle
=\langle u(x),\varphi (x)\rangle .
\end{equation}
\end{definition}
\begin{definition}
Let $f\in {\mathcal K_{\mu}'}(\RR_+)$. We say that the distribution $f$ has quasi-asymptotic behavior (\textit{the quasi-asymptotics}) of degree $m\in\R$ at infinity  with respect to $L$ in $\mathcal{K}_{\mu}'(\RR_+)$ if there exists $u\in \mathcal{K}_{\mu}'(\RR_+)$
such that for every test function $\varphi \in \mathcal{K}_{\mu}(\RR_+)$ the following limit holds:
\begin{equation} \label{qbeh_infty}
\lim_{\lambda \to \infty } \Big\langle \frac{f(\lambda
x)}{\lambda ^{m } L(\lambda )} ,\ \varphi (x)\Big\rangle
=\langle u(x),\varphi (x)\rangle .
\end{equation}
\end{definition}
\noindent We also use the following convenient notation for the
quasi-asymptotic behavior of degree $m \in {\Bbb R}$ at a point $x_{0} \in {\Bbb R}$ (resp. at infinity)
with respect to $L$ in $\mathcal{K}_{/mu}'(\RR_+)$
\[f(x_{0} +\varepsilon x) \sim \varepsilon ^{m } L(\varepsilon )u(x)\quad {\rm as} \quad \varepsilon \to 0^{+} ,\quad \Big({\rm resp.} \quad f(\lambda x) \sim \lambda ^{m }
L(\lambda )u(x) \quad {\rm as} \quad \lambda \to +\infty\Big) \]
which should always be interpreted in the weak topology of
$\mathcal{K}_{\mu}'(\RR_+)$, i.e., in the sense of \eqref{qbeh_x_0} (resp. \eqref{qbeh_infty}).

We underline that $u$ cannot have an arbitrary form; indeed, it
must be homogeneous with degree of homogeneity $m $, i.e.,
$u(ax)=a^{m } u(x)$, for all $a>0 $, see \cite{PST,
VDZ}.

We introduce an additional concept from quasi-asymptotic analysis, namely the notion of
 {quasi-asymptotic boundedness}, as defined in \cite{Vindas5}.
\begin{definition}
    Let $f\in \mathcal{K}_{\mu}'(\RR_+)$. The distribution $f$ is \emph{quasiasymptotically
bounded} of degree $m\in \R$ at a point $x_0\in\mathbb{R}$ with respect to $L$ in $\mathcal{K}_{\mu}'(\RR_+)$ if for every test function $\varphi \in \mathcal{K}_{\mu}(\RR_+)$ the following holds:
\begin{equation*} \label{eq10} \langle
f(x_0+\varepsilon x),\,\varphi(x)\rangle=O(\varepsilon^m
L(\varepsilon)) \,\,{\mbox {as}}\,\,\varepsilon\rightarrow\
0^+\ .
\end{equation*}
\end{definition}

\noindent This implies that for $x$ sufficiently close to $x_0$, there exists a constant $C>0$ such that
$$|\langle
f(x_0+\varepsilon x),\,\varphi(x)\rangle|\leq C\varepsilon^m
L(\varepsilon), \varphi \in \mathcal{K}_{\mu}(\RR_+).$$

The notion of quasi-asymptotic boundedness of degree $m\in \R$ at infinity with respect to $L$ in $\mathcal{K}_{\mu}'(\RR_+)$ is defined in an analogous manner.


\begin{remark}
We may also consider the quasi-asymptotics in other distribution
spaces. Let $\mathcal{A}(\R_+)$ denote a space of test functions on $\R_+$, and let $\mathcal{A}'(\R_+)$ be its dual space.The relation $f(x_{0} +\varepsilon x)\sim\varepsilon^{m } L(\varepsilon )u(x)\,\, {\rm as}\,
\,\varepsilon \to 0^{+} \ {\rm in\; } {\mathcal A'}({\Bbb R_+})$
means that \eqref{qbeh_x_0} holds for each $\varphi \in
{\mathcal A}({\Bbb R_+})$ (resp. for the quasi-asymptotics at
infinity in ${\mathcal A'}({\Bbb R_+})$).
\end{remark}

\subsection{Tauberian-type results}\label{Sec Abelian and Tauberian-type results}

The study of  Tauberian-type results for the Hankel transform on various distribution spaces has been carried out in \cite{Hankel1,Hankel2, Zem3}.

%



For the proof of the next theorem, we need the following result from \cite{MAM}: $L$ is a slowly varying function at the origin if and only if there exist measurable functions $u$ and $w$ defined on the interval $(0, A]$, where $u$ is bounded and has a finite limit at 0, and $\omega$ is continuous on $[0, A]$ with $\omega(0) = 0$, such that the subsequent representation holds for $L(x)$ within the interval $(0, A]$:
$$L(x) = \exp{\left(u(x)+\int_{x}^{A}\frac{\omega(\xi)}{\xi}d\xi\right)}, \qquad  x\in(0, A] .$$
When considering functions with slow variation at the origin, the condition $\xi^{-1}\omega(\xi)\in L^{1}([1, \infty))$ implies the existence of positive constants $C_1$ and $C_2$ such that the following inequalities hold for $x > 1$:
\begin{equation}\label{ex const}
C_1 < L (x) < C_2.
\end{equation}


\begin{theorem}\label{tbound} Let   $f\in \mathcal K_{\mu}'(\RR_+)$ and (\ref{ex const}) hold for $L$, and let the distribution $f$ is quasi-asymptotically bounded of degree $m\in\RR$ at the origin with respect to $L$ in $\mathcal K_{\mu}'(\RR_+)$ , i.e.,
\begin{equation}\label{qa bounded at 0}
|\langle f (\varepsilon x), \varphi (x)\rangle|\leq C_{\varphi}\varepsilon^{m}L(\varepsilon), \quad\quad\mbox{as}\quad \varepsilon\to 0^+,
\end{equation}
for every $\varphi\in\mathcal K_{\mu}(\RR_+)$, where $C_{\varphi} > 0$ depends on $\varphi$. Then the FrHT  $H^\alpha_{\mu}(f)$ ($\mu$ is fixed), is bounded at the origin, i.e., there exists  $C>0$ such that
$$\Big|\Big\langle e^{ic_1 (\frac{\xi}{\varepsilon})^2/2}H_{\mu}^{\alpha}f\Big(\frac{\xi}{\varepsilon}\Big), \varphi(\xi)\Big\rangle\Big|\leq C \varepsilon^{m+1}L(\varepsilon).$$
\end{theorem}

\begin{proof}
We employ \eqref{FrHT veza HT}, \eqref{qa bounded at 0}, \eqref{ex const} and (\ref{uslnep}) to conclude that there exists  $C>0$ such that
\begin{eqnarray*}
& & \frac{1}{\varepsilon^{m+1}L(\varepsilon)}\Big|\Big\langle e^{ic_1 (\frac{\xi}{\varepsilon})^2/2} H_{\mu}^{\alpha}f \Big(\frac{\xi}{\varepsilon}\Big),\varphi(\xi)\Big\rangle \Big|\\
& = & \Big|\frac{C_{\alpha,\mu}}{c_2}\frac{1}{\varepsilon^mL(\varepsilon) }\Big\langle e^{-ic_1 (\frac{\varepsilon t}{c_2})^2/2}f\Big(\frac{\varepsilon t}{c_2}\Big),H_{\mu}\varphi(t)\Big\rangle\Big|
\leq C\cdot C_{\alpha,\mu}\gamma^\mu_r(H_{\mu} \varphi)<\infty.
\end{eqnarray*}

\end{proof}
Before we state the Tauberian theorem for the quasi-asymptotic behavior, we recall the definition of the space $\mathcal B_{\mu,b}(\RR_+), \mu\geq-1/2,$ given in \cite[Section 2]{Zem2}. Let $b$ be a positive real number. The space $\mathcal B_{\mu,b}(\RR_+)$ is the space of all smooth complex-valued functions $\psi(x), \ x\in\R_+,$ such that $\psi(x)=0, \  b<x<\infty$, and $(x^{-1}D)^k x^{-\mu-1/2}\psi(x)$ is bounded on $\R_+$. The topology of $\mathcal B_{\mu,b}(\RR_+)$ is generated by the  seminorms
\begin{equation}\label{nsw1}
\gamma^\mu_{k}(f)=\sup_{x\in \R_+ \atop k\leq \mu}|(x^{-1}D)^k\big(x^{-\mu-\frac{1}{2}}f(x)\big)|<\infty\;   \ \ k\in\N_0.
\end{equation}
\noindent It should be noted that for $b<c$, $\mathcal B_{\mu,b}(\RR_+)\subset \mathcal B_{\mu,c}(\RR_+)$.

\noindent For any fixed $\mu$, $\mathcal B_{\mu}(\RR_+)$ denotes the strict inductive limit of $\mathcal B_{\mu,b}(\RR_+)$, where $b$ traverses a monotonically increasing sequence of positive numbers and tends to infinity. Both $\mathcal B_{\mu,b}(\RR_+)$ and $\mathcal B_{\mu}(\RR_+)$ are sequentially complete, see \cite{Zem2} and cf. \cite[Lemma 2]{Zem1}.\\



Now, we can formulate the Tauberian theorem for the FrHT.

\begin{theorem}\label {te1} Assume that $f\in \mathcal{K}_{\mu}'(\R_+)$ is of slow growth at infinity and locally integrable on $\R_+$, defining a regular element of $\mathcal{K}_{\mu}'(\R_+)$. Let:
\begin{enumerate}
\item [(i)] The limit
\begin{equation}\label{limit1}
\lim_{\varepsilon \to 0^{+} }
\frac{e^{ic_1 (\frac{\xi}{\varepsilon})^2/2} H_{\mu}^{\alpha} f (\frac{\xi}{\varepsilon})} {\varepsilon ^{m +1} L(\varepsilon
)}=M_{\xi }<\infty,
\end{equation}

\noindent exists for all $\xi\in\R_+$.

\item [(ii)] There exist constants $N>0$, $C=C(\alpha)>0$, and $0<\varepsilon_0\leq 1$ such that
\begin{equation}\label{ogranicuvanje1}
\frac{|e^{ic_1 (\frac{\xi}{\varepsilon})^2/2} H_{\mu}^{\alpha} f (\frac{\xi}{\varepsilon})|} {\varepsilon ^{m+1} L(\varepsilon
)}\leq C |\xi|^{N+\mu+\frac{1}{2}},
\end{equation}
\noindent for all $\xi\in\R_+$ and $0< \varepsilon \leq
\varepsilon_0$.
\end{enumerate}
Then $f$ has the quasi-asymptotic behavior of degree $m\in\RR$ at $0^+$ in $\mathcal{K}_{\mu}'(\RR_+)$.

\end{theorem}

\begin{proof} In view of \eqref{limit1} and \eqref{ogranicuvanje1}, we may assume that the function defined by $\xi\mapsto M_{\xi}, \xi\in\R_+ $, is measurable and satisfies the following
estimate
$|M_{\xi} |\leq C|\xi|^{N+\mu+\frac{1}{2}}, \xi\in\R_+,$ where $C=C(\alpha)>0$.
By the Banach–Steinhaus theorem, it is sufficient to verify that $ \lim_{\varepsilon \to 0^{+} }\langle\frac{e^{-ic_1(\varepsilon x)^2/2}f(\varepsilon x)}{\varepsilon ^{m } L(\varepsilon )}, \varphi \rangle$ exists for all $\varphi$ from a dense subspace of $\mathcal K_{\mu}(\RR_+)$. Since $\mathcal B_{\mu}(\RR_+)$ is dense in $\mathcal K_{\mu}(\RR_+)$,
by the  use of the inversion formula \eqref{inverse}, it is sufficient to show that for $\varphi\in\mathcal B_\mu$ there holds

\begin{align*}
\mathop{{\rm lim}}\limits_{\varepsilon \to 0^{+} } \Big\langle \frac{e^{-ic_1(\varepsilon x)^2/2}f(\varepsilon x)}{\varepsilon ^{m } L(\varepsilon )} ,\varphi (x)\Big\rangle &=\mathop{{\rm lim}}\limits_{\varepsilon \to 0^{+} } \Big\langle \frac{e^{-ic_1x^2/2}f(x)}{\varepsilon ^{m +1} L(\varepsilon )} ,\varphi \Big(\frac{x}{\varepsilon}\Big)\Big\rangle \\
& =\mathop{{\rm lim}}\limits_{\varepsilon \to 0^{+} } \Big\langle \frac{e^{-ic_1x^2/2}}{\varepsilon ^{m+2 } L(\varepsilon )} \int_{0}^\infty (H_{\mu}^{\alpha}f)\Big(\frac{\xi}{\varepsilon}\Big)\overline{K}_{\alpha}(x, \frac{\xi}{\varepsilon})d\xi, \varphi\Big(\frac{x}{\varepsilon}\Big)\Big\rangle\\
&=\mathop{{\rm lim}}\limits_{\varepsilon \to 0^{+} } C^*_{\alpha,\mu}\Big\langle \frac{e^{ic_1 (\frac{\xi}{\varepsilon})^2/2}H_{\mu}^{\alpha}f(\frac{\xi}{\varepsilon})}{\varepsilon ^{m+1 } L(\varepsilon )},H_\mu\varphi(c_2\xi)\Big\rangle\\
&=\mathop{{\rm lim}}\limits_{\varepsilon \to 0^{+} } C^*_{\alpha,\mu}c_2^m \Big\langle \frac{e^{ic_1 (\frac{\xi}{c_2\varepsilon})^2/2}H_{\mu}^{\alpha}f(\frac{\xi}{c_2\varepsilon})}{(c_2\varepsilon) ^{m+1 } L(c_2\varepsilon )\frac{L(\varepsilon)}{L(c_2\varepsilon)}},H_{\mu}\varphi(\xi)\Big\rangle .
\end{align*}
The assumptions in \eqref{limit1} and
\eqref{ogranicuvanje1}, allow us to apply the Lebesque's dominated convergence
theorem. Applying the estimate \eqref{limL1}, we obtain
\begin{equation}\label{Int fin}
\mathop{\lim }\limits_{\varepsilon \to 0^{+} } \langle \frac{e^{-ic_1(\varepsilon t)^2/2}f(\varepsilon t)}{\varepsilon ^{m} L(\varepsilon )} ,\varphi (t)\rangle =C^*_{\alpha,\mu}c_2^m \langle M_\xi,H_{\mu}\varphi(\xi)\rangle.
\end{equation}
Note that
$|M_\xi|=O(|\xi |^{N+\mu+\frac{1}{2}} )$ for some \textit{$N>0$
}, and that $|H_{\mu}\varphi(\xi)|=O(|\xi|^{-n+\mu+\frac{1}{2}}
)$ holds for all $n>0$, whenever $\varphi \in {\mathcal
K_{\mu}}(\Bbb R_+)$. It follows that the last integral in (\ref{Int fin}) converges absolutely. Finally, we conclude that
$$\lim_{\varepsilon \to 0^{+} } \langle
\frac{e^{-ic_1(\varepsilon t)^2/2}f(\varepsilon t)}{\varepsilon ^{m } L(\varepsilon )}
,\varphi (t)\rangle $$ exists for all $\varphi \in {\mathcal
K_{\mu}}(\Bbb R_+)$. 
\end{proof}

\section{Initial and final value theorems for the FrHT on distribution spaces}\label{Sec IVT and FVT}

The final value theorem (FVT) for the FrHT provides a way to find the behavior of $f(x)$, as $x\rightarrow\infty$, using the behavior of its FrHT $H_{\mu}^{\alpha}f(\xi)$, as $\xi\rightarrow 0^+$. The initial value theorem (IVT) for the FrHT provides a way to recover $f(x)$, as $x\rightarrow 0^+$, from the asymptotic behavior of its FrHT, $H_{\mu}^{\alpha}f(\xi)$, as $\xi\rightarrow\infty$.

The initial and final value theorems for FrHT in the Zemanian space $\mathcal{K}_{\mu}(\RR_+)$ are given  in \cite{Gudadhe}.


\subsection{An initial value theorem for the FrHT} After making some changes and utilizing the proof of \cite[Lemma 1]{Zem3}, the IVT for the FrHT \cite[Theorem 2.1]{Gudadhe} takes on the following form:

\begin{lemma}\label{IVT lema}
Let $3/2 < \eta < 2+\mu$. Let $f(x)$ be a measurable function that is Lebesgue integrable on every interval of the form $X<x<\infty$ ($X>0$), and suppose that $f(x)$ satisfies the following conditions:
\begin{itemize}
\item[(i)] $f(x)\rightarrow 0$ as $x\rightarrow\infty$,
\item[(ii)] $x^{\eta-1/2}f(x)$ is absolutely continuous on $0\leq x <\infty$.
\end{itemize}
Set
$$ \lim_{x\rightarrow 0^+}e^{-\frac{i}{2}x^2 c_1}x^{\eta-\frac{1}{2}}f(x)=\rho\frac{ c_2^{\frac{3}{2}-\eta}}{C_{\alpha,\mu}}.$$
Then
$$\lim_{\xi\rightarrow\infty}e^{i\xi^2c_1/2}\xi^{3/2-\eta}H^\alpha_{\mu} f(\xi)=\rho\; H(\mu,\eta),$$
where
\begin{equation}\label{funkcija H}
H(\mu,\eta)=\frac{\Gamma(1+\mu/2-\eta/2)}{2^{\eta-1}\Gamma(\mu/2+\eta/2)}.
\end{equation}
\end{lemma}

\begin{remark}
It should be noted that $H^\alpha_{\mu} f(\xi)$ exists for every positive value of $\xi$ since $f(x)=O(x^{1/2-\eta})$ as $x\rightarrow 0^+$, $J_{\mu}(x)=O(x^{\mu})$ as $x\rightarrow 0^+$ and $\sqrt{x}J_{\mu}(x)=O(1)$ as $x\rightarrow \infty$.
\end{remark}

\begin{proof}
Application of
\begin{equation}\label{pomocna}
    \int_{0}^{\infty}z^{1-\eta}J_{\mu}(z)dz=H(\mu,\eta),\quad 1<\nu<2+\mu,\quad\quad \mbox{(see \cite[p.22,(7)]{Erdelyi})},
\end{equation}
with $z=x\xi c_2$ and $Z=X\xi c_2$ leads to

\begin{eqnarray}
& & |e^{i\xi^2c_1/2}\xi^{\frac{3}{2}-\eta}H^\alpha_{\mu} f(\xi)-\rho H(\mu,\eta)|\nonumber  \\
& = & \Big{|} e^{i\xi^2c_1/2}\xi^{\frac{3}{2}-\eta}\int_0^\infty C_{\alpha,\mu}e^{-i(\frac{x^2+\xi^2}{2}c_1)}\sqrt{x\xi c_2}J_{\mu}(x\xi c_2)f(x)dx - \rho \xi c_2 \int_0^\infty(x\xi c_2)^{1-\eta}J_{\mu}(x\xi c_2)dx\Big{|}\nonumber \\
& \leq & \xi  \int_0^\infty \Big| (x\xi c_2)^{1-\eta} J_{\mu}(x\xi c_2) \Big|\; \Big| C_{\alpha,\mu}e^{-\frac{i}{2}x^2 c_1}(xc_2)^{\eta-\frac{1}{2}}f(x)-\rho c_2 \Big| dx \nonumber \\
& \leq & \sup_{0<x<X}\Big| C_{\alpha,\mu}e^{-\frac{i}{2}x^2 c_1}(xc_2)^{\eta-\frac{1}{2}}f(x)-\rho c_2\Big|\;\frac{1}{c_2}  \int_0^Z \Big| z^{1-\eta}J_{\mu}(z)\Big| dz  \label{konvergencija1}\\
& + &  \xi^{\frac{3}{2}-\eta}\int_{X}^\infty \Big| \sqrt{x\xi c_2}J_{\mu}(x\xi c_2)\Big|\;\Big|C_{\alpha,\mu}e^{-\frac{i}{2}x^2c_1}f(x)-\rho c_2(xc_2)^{\frac{1}{2}-\eta}\Big|dx\label{konvergencija2} .
\end{eqnarray}

\noindent Utilizing $J_{\mu}(z)=O(z^{\mu})$ as $z\rightarrow 0^+$ and $1-\eta+\mu >-1$, one concludes that the integral in (\ref{konvergencija1}) converges. Let $\epsilon >0$. By choosing $X$ small enough, the second term in (\ref{konvergencija1}), which does not depend on $\xi$, can be made less than $\epsilon/2$. Since $\sqrt{z}J_{\mu}(z)$ is bounded on $\R_+$, with $X$ fixed, the first term in (\ref{konvergencija2}) is less than $\epsilon/2$, for all sufficiently large $\xi$. Thus, the proof of the lemma is complete.

\end{proof}

The next step is to extend Lemma \ref{IVT lema} to distributions. We utilize the proof of \cite[Theorem 1]{Zem3}.

In extending the results, we need the following fact \cite[Section 3.3]{Zem4}: If $f\in\mathcal{E}'(\RR_+)$, there exists a constant $C$ and a non-negative integer $r$ such that for every $\varphi\in\mathcal{D}(\RR_+)$
\begin{equation}\label{difer}
|\langle f, \varphi\rangle|\leq C\sup_{x\in\R_+}|D_x^r\varphi(x)|.
\end{equation}
The smallest non-negative integer $r$ such that (\ref{difer}) holds is called the order of $f$.

The following proposition is required for the subsequent proofs.

\begin{proposition}(cf. \cite[Theorem 2]{Zem1})\label{ocena pomocne}
If $f\in\mathcal{E}'(\RR_+)$, then
$$F(\xi)=\langle f(x), C_{\alpha,\mu}e^{-i(\frac{x^2+\xi^2}{2}) c_1}\sqrt{x\xi c_2}J_{\mu}(x\xi c_2)\rangle,\quad\xi\in\RR_+,$$
is a smooth function on that satisfies the inequality of the form
$$ |F(\xi)| \leq \left\{\begin{array}{ll}
           K\xi^{\mu+1}, & \mbox{if}\quad 0<\xi<1\\
            K\xi^s,  & \mbox{if}\quad 1<\xi<\infty,
         \end{array}
         \right. $$
where $K$ and $s$ are sufficiently large constants.
\end{proposition}

Before formulating the next theorem, we remind a reader that compactly supported distributions $\mathcal{E}'(\RR_+)$ are elements of $\mathcal{K}_{\mu}'(\RR_+)$.

\begin{theorem}\label{IV Thm}
Assume that $f\in \mathcal K'_{\mu}(\RR_+)$ can be decomposed into $f=g+h$, where $g$ is an ordinary function satisfying the hypothesis of Lemma \ref{IVT lema}, and $h\in \mathcal E'(\RR_+)$ is of order $r$. Let  $\mu\geq 1/2$ and $\eta\in\R$ be such that $3/2+r<\eta<2+\mu$. Then
\begin{equation*}
\lim_{\xi\to\infty}e^{i\xi^2 c_1/2}\xi^{3/2-\eta}H_{\mu}^{\alpha}f(\xi)=\rho\; H(\mu,\eta),
\end{equation*}
where $H(\mu,\eta)$ is defined by \eqref{funkcija H}.
\end{theorem}

\begin{proof}
In view of Proposition \ref{ocena pomocne}, the function
\begin{equation}\label{pomocna1}
   H(\xi)=\langle h(x),C_{\alpha,\mu}e^{-i(\frac{x^2+\xi^2}{2}) c_1}\sqrt{x\xi c_2}J_{\mu}(x\xi c_2)\rangle,\quad \xi\in\RR_+
\end{equation}
is smooth and exhibits slow growth as $\xi\rightarrow\infty$.
We can now relate the growth of $H(\xi)$ to the order of $h$.

\noindent Let $\lambda(x)\in \mathcal D(\RR_+)$ such that $\lambda(x)\equiv 1$ on a neighborhood of the support of $h$. Then, by (\ref{difer}), we obtain

\begin{eqnarray*}
& & |H(\xi)| \leq C(\alpha) \sup_{x\in\R_+}|D_{x}^{r}(\lambda(x) e^{-i(\frac{x^2+\xi^2}{2}) c_1}\sqrt{x\xi c_2}J_{\mu}(x\xi c_2))|\\
& \leq & C_1(\alpha) \sup_{x\in\R_+}
\sum_{n=0}^{r}\sum_{k=0}^{n}\binom{r}{n}\binom{n}{k} \Big |x^{n-k} D_{x}^{r-n}\lambda(x)\Big| \Big| \xi^k D_z^k (\sqrt{z c_2}J_{\mu}(z c_2) ) \Big |.
\end{eqnarray*}

\noindent where $z=x\xi$. There is the constant $C_{r-n}$ such that $|x^{n-k}D_{x}^{r-n}\lambda(x)|\leq C_{r-n}$.
From the series expansion for $J_{\mu}$ \cite[p.134]{Jahnke},  we see that $D_{z}^{n}(\sqrt{zc_2}J_{\mu}({z c_2}))=O(1)$ as $z\to 0^{+}$ for $n=0,1,\dots, r$ since $r<\mu+1/2$. From the differential formula \cite[p.154]{Jahnke} and the asymptotic behavior \cite[ p.147]{Jahnke} of $J_{\mu}(z)$ it follows that $D_{z}^{n}(\sqrt{zc_2}J_{\mu}({z c_2}))=O(1)$ as $z\to \infty$. Since $D_{z}^{n}(\sqrt{zc_2}J_{\mu}({z c_2}))$ is continuous on $\RR_+$, it follows that it is also bounded for $n=0,1,...,r$. Hence, for $\xi>1$, $|H(\xi)|<K\xi^r$ for some sufficiently large constant $K=K(\alpha)$, which depends of $\alpha$.

\noindent As $3/2+r-\eta<0$, we have that $e^{i\xi^2 c_1/2}\xi^{3/2-\eta}H(\xi)\to 0$, as $\xi\to \infty$. On the other hand, since the support of $h\in\mathcal E'(\RR_+)$ is a compact subset of $\RR_+$, it follows that
$$\lim_{x\to 0^{+}}e^{-i\frac{x^2}{2}c_1}x^{\eta-1/2}h(x)=0.$$
Using Lemma \ref{IVT lema}, we complete the proof.

\end{proof}

\subsection{A final value theorem for the FrHT} After making some changes and using the proof of \cite[Lemma 2]{Zem3}, the FVT for the FrHT \cite[Theorem 3.1]{Gudadhe} takes on the following form:

\begin{lemma}\label{FVT lema}
Let $3/2 < \eta < 2+\mu$. Suppose $f$ satisfies the following conditions:
\begin{itemize}
\item[(i)] On every interval of the form $0<x<X$ ($X<\infty$), $x^{\mu+1/2}f(x)$ is Lebesgue integrable.
\item[(ii)] There exists a complex number $\delta$ such that
$$ \lim_{x\rightarrow \infty}e^{-\frac{i}{2}x^2 c_1}x^{\eta-\frac{1}{2}}f(x)=\delta\frac{ c_2^{\frac{3}{2}-\eta}}{C_{\alpha,\mu}}$$
\end{itemize}

\noindent Then
$$\lim_{\xi\rightarrow 0^+}e^{i\xi^2c_1/2}\xi^{3/2-\eta}H^\alpha_{\mu} f(\xi)=\delta\; H(\mu,\eta),$$
where $H(\mu,\eta)$ is defined by (\ref{funkcija H}).
\end{lemma}

\begin{remark}
It should be noted that since $J_{\mu}(x)=O(x^{\mu})$ as $x\rightarrow 0^+$ and $\sqrt{x}J_\mu(x)=O(1)$ as $x\rightarrow \infty$, the conditions on $f(x)$ ensure that $H^\alpha_{\mu} f(\xi)$ exists for every positive value of $\xi$.
\end{remark}

\begin{proof}
Utilizing (\ref{pomocna}) and the fact that
$$|\sqrt{z}J_{\mu}(z)|<C_{\mu}z^{\mu+1/2},\quad 0<z<\infty,$$
where $C_{\mu}$ is sufficiently large constant, we may write

\begin{eqnarray}
& & |e^{i\xi^2c_1/2}\xi^{\frac{3}{2}-\eta}H^\alpha_{\mu} f(\xi)-\delta H(\mu,\eta)|\nonumber  \\
& \leq & \xi^{\frac{3}{2}-\eta}\int_{0}^X \Big| \sqrt{x\xi c_2}J_{\mu}(x\xi c_2)\Big|\;\Big|C_{\alpha,\mu}e^{-\frac{i}{2}x^2c_1}f(x)-\delta c_2(xc_2)^{\frac{1}{2}-\eta}\Big|dx \nonumber \\
& + & \int_{X\xi}^\infty \sup_{0<x<X}\Big| C_{\alpha,\mu}e^{-\frac{i}{2}x^2 c_1}(xc_2)^{\eta-\frac{1}{2}}f(x)-\delta c_2\Big| \Big| (x\xi c_2)^{1-\eta}J_{\mu}(x\xi c_2)\Big| d(x\xi)  \nonumber \\
& \leq & C_{\mu} \xi^{2+\mu-\eta}\int_{0}^X (xc_2)^{\mu+\frac{1}{2}}\;\Big|C_{\alpha,\mu}e^{-\frac{i}{2}x^2c_1}f(x)-\delta c_2(xc_2)^{\frac{1}{2}-\eta}\Big|dx \label{konvergencija 3} \\
& + & \frac{1}{c_2}\int_0^\infty \Big| z^{1-\eta}J_{\mu}(z)\Big| dz\; \sup_{0<x<X}\Big| C_{\alpha,\mu}e^{-\frac{i}{2}x^2 c_1}(xc_2)^{\eta-\frac{1}{2}}f(x)-\delta c_2\Big| .  \label{konvergencija 4}
\end{eqnarray}

\noindent The term in integral (\ref{konvergencija 3}) is independent of $\xi$ and $x^{\mu+1/2}f(x)$ is integrable according to assumption $(i)$. Hence, the integral in (\ref{konvergencija 3}) converges if $x^{\mu-\eta+2}$ converges. Since $\mu-\eta+2>0$, we conclude that the integral in (\ref{konvergencija 3}) converges. It is easy to see that the integral in (\ref{konvergencija 4}) converges. Let $\epsilon >0$. By choosing $X$ sufficiently large, (\ref{konvergencija 4}) can be made less than $\epsilon/2$. Next, there exists $\Xi>0$ such that (\ref{konvergencija 3}) is less than $\epsilon/2$ for $0<\xi<\Xi$. This completes the proof of the lemma.
\end{proof}

The next step is to extend Lemma \ref{FVT lema} to distributions. We utilize the proof of \cite[Theorem 2]{Zem3}.

\begin{theorem}\label{FV Thm}
Let $3/2<\eta<\mu+2$. Assume that $f\in\mathcal{K}'_{\mu}(\RR_+)$ can be decomposed into $f=g+h$, where $g$ is an ordinary function satisfying the hypothesis of Lemma \ref{FVT lema} and $h\in\mathcal E'(\RR_+)$. Then

$$\lim_{\xi\to 0^+}e^{i\xi^2 c_1/2}\xi^{3/2-\nu}H_{\mu}^{\alpha}f(\xi)=\delta H(\mu,\eta),$$
where $H(\mu,\nu)$ is defined by \eqref{funkcija H}.
\end{theorem}

\begin{proof}

Using Proposition \ref{ocena pomocne}, we infer that the function defined by (\ref{pomocna1}) is a smooth function for $0<\xi<\infty$, and satisfies the inequality $|\xi^{3/2-\eta}H(\xi)|\leq K\xi^{2+\mu-\eta}$ for $0<\xi<1$, where $K=K(\alpha)$ is a sufficiently large constant. Since $2+\mu-\nu>0$, we have
$\lim_{\xi\rightarrow 0^+}K\xi^{2+\mu-\eta}= 0.$
By the assumption, $\lim_{x\to\infty}e^{-i\frac{x^2}{2}c_1}x^{\eta-1/2}h(x)=0$. Finally, utilizing Lemma \ref{FVT lema}, we conclude the proof.

\end{proof}

\section{Declaration}
All authors contributed equally to this work, and have given a consent to participate and consent for publication.
No funding was received.

\subsection{Conflict of interest}
All authors declare that they have no conflicts of interest.

\subsection{Data availability}
The authors confirm that the data supporting the findings of this study are available within the article.

\subsection{Acknowledgment} 
%

The first and the fourth authors are supported by the Serbian Academy of Sciences and Arts, project F-10; The second author is supported by the Science Fund of the Republic of
Serbia, $\#$GRANT No. 2727, {\it Global and local analysis of operators and
distributions} - GOALS and by the Ministry of Science, Technological Development and Innovation of the Republic of Serbia (Grants No. 451-03-137/2025-03/ 200169); The first and the third authors are supported by the Republic of Srpska Ministry for Scientific and Technological Development and Higher Education, project \textit{Modern methods of time-frequency analysis and fixed point
theory in spaces of generalized functions with applications}, No. 19.032/431-1-22/23.



\end{document}